\documentclass{amsart}
\usepackage{latexsym,amsmath,amsthm,amssymb}

\newtheorem{theorem}{Theorem}[section]
\newtheorem{lemma}[theorem]{Lemma}
\newtheorem{corollary}[theorem]{Corollary}
\newtheorem{proposition}[theorem]{Proposition}

\theoremstyle{remark}
\newtheorem{example}[theorem]{Example}
\newtheorem{remark}[theorem]{Remark}

\theoremstyle{definition}
\newtheorem{definition}[theorem]{Definition}

\newcommand{\dR}{\ensuremath{\mathbb{R}}} 
\newcommand{\R}{\dR}

\begin{document}

\title{On global H{\"o}lder estimates for optimal transportation}

\author{Alexander V. Kolesnikov}

\begin{abstract}
We generalize a well-known result of L.~Caffarelli
on Lipschitz estimates for optimal transportation $T$ between uniformly log-concave probability
 measures.
Let $T : \R^d \to \R^d$ be an optimal transportation pushing forward
$\mu = e^{-V}dx$ to $\nu = e^{-W}dx$.
Assume that 1) the second differential quotient of $V$
can be estimated from above by a power function,
2) modulus of convexity of $W$ can be estimated  from below by $A_q |x|^{1+q}$, $q \ge 1$.
Under these assumptions we show that $T$
is globally H{\"o}lder with a dimension-free coefficient.
In addition,  we study optimal transportation $T$ between  $\mu$
and the uniform measure on a bounded convex  set $K \subset \R^d$.
We get estimates for the Lipschitz constant of $T$ in terms of $d$, $\mbox{diam(K)}$ and
$D V, D^2 V$.
\end{abstract}

\maketitle
Keywords: optimal transportation,
Monge-Amp{\`e}re equation, H{\"o}lder apriori estimates, convex bodies, log-concave measures, isoperimetric
and concentration inequalities

\section{Introduction}

According to a well-known result of L.~Caffarelli \cite{Caf}
any optimal transportation mapping $T$ pushing forward
the standard Gaussian measure $\gamma$ to a
probability measure $e^{-W} \cdot \gamma$ with convex $W$ is $1$-Lipschitz.
This remarkable observation allows to recover many interesting results on analytic properties
of uniformly convex measures. For instance, the Bakry-Emery condition for the log-Sobolev inequality ("flat"
part) \cite{bakre85dh}, Bakry-Ledoux comparison theorem \cite{BakLed},
some correlation inequalities \cite{Caf}, \cite{GH2}.
Some recent generalizations of \cite{Caf}  can be found in \cite{Vald}.

The estimates of this type go back to A.V.~Pogorelov \cite{Pog}
(see also "Pogorelov lemma" in  \cite{Guit}).
According to his result, any the smooth solution of the Monge-Amp{\`e}re equation
$\det D^2 \varphi=1$ in a smooth domain $\Omega \subset \R^d$ satisfies
$$
C_1(\Omega,\Omega') \le D^2 \varphi(x) \le C_2(\Omega,\Omega'),
$$
$x \in \Omega'$, where $\Omega' \subseteq \Omega$ is another domain.
The H{\"o}lder estimates play
fundamental role in the regularity theory of partial differential equations.
The overview of results on a priori estimates for
fully nonlinear differential equations see in \cite{Kryl2}, \cite{CafCab}.
For a more special case of the Monge-Amp{\`e}re equation the readers are advised to consult
\cite{Guit}, \cite{Urbas}, \cite{Vill}.

On the other hand, mass transportation method is widely used in probability
for establishing diverse isoperimetric and concentration properties.
Given a "good" measure $\mu$ with known isoperimetric properties (e.g. Gaussian, product measures)
and another measure $\nu$ one can try to transform
$\mu$ into $\nu$ and deduce the desired information about $\nu$
from the properties of the mass transport (Lipschitz or H{\"o}lder estimates).
Let us indicate some typical situations

\newpage
\begin{itemize}
\item[1)] $\mu$ is Gaussian and $\nu$ is another
product measure (see \cite{Tal}, \cite{Led})
\item[2)] $\nu$ is the uniform measure on
a convex set
and $\mu$ is log-concave with the same modulus of convexity
(see, for instance, \cite{BobLed}, \cite{Led}, \cite{MilSod},  and references therein)
\item[3)] $\mu$ is Gaussian and $D^2 W \ge C$ (situation of the Caffarelli theorem).
\end{itemize}

The paper is organized as follows.
In Section 2 we consider the  optimal  transportation $T$ pushing forward $e^{-V} dx$ to $e^{-W} dx$. Assume that $V$ and $W$ satisfy
$$
V(x+ y) + V(x-y) -2V(x) \le A_p |y|^{p+1},
$$
$$
W(x+ y) + W(x-y) -2W(x) \ge A_q |y|^{q+1}
$$
with $0 \le p \le 1 \le q$ and some positive $A_p, A_q$.
We show that $T$ is globally $\frac{p+1}{q+1}$-H{\"o}lder:
$$
|T(x)-T(y)| \le C |x-y|^{\frac{p+1}{q+1}},
$$
with $C$ depending only on $p, q, A_p, A_q$.

In particular, this implies the following. Let $\nu = e^{-W} dx$  be a probability measure satisfying
$$
W(x+ y) + W(x-y) -2W(x) \ge C_{\beta} |y|^{2 \beta}
$$
with $\beta \ge 1$.
Then there exists $C$ depending only on $\beta, C_{\beta}$ such that
 $\nu$ satisfies the following concentration inequality:
$$
\nu(B_r) \ge 1-e^{-C|r|^{2\beta}},
$$
for every $B \subset \R^d$ with $\nu(B) \ge \frac{1}{2}$, $B_r = \{y: |x-y| \le r, x \in B\}$.
This is a previously known consequence of
a modified log-Sobolev inequality proved by Bobkov and Ledoux \cite{BobLed}.
The Bobkov-Ledoux result has been
generalized in \cite{CGH} by transportation arguments.
The corresponding isoperimetric inequalities
have been proved in \cite{MilSod} by localization techniques.

In Sections 3-4 we prove estimates of the following type:
$$
\|D^2 \varphi \| \le f(\nabla \varphi),
$$
where $f$ is a non-negative function and $\nabla \varphi$ is the optimal
mass transport pushing forward a smooth measure $\mu$ to
 $\nu = \frac{1}{\lambda(K)} \cdot \lambda|_{K}$, where $K \subset \R^d$ is convex
and $\lambda$ is the Lebesgue measure.
In particular, we show that
any uniform measure $\nu$ on a bounded convex set $K$ is a Lipschitz image of the standard Gaussian measure
which Lipschitz constant does not exceed
$ C \sqrt{d} \cdot \mbox{diam} (K)$ with some universal $C$.

 After  publishing online a preliminary version of this paper, the author
obtained a remark from the participants of the Convex Analysis seminar in Tel-Aviv University.
 They noticed that the proof of Proposition \ref{Stein-lem}
 can be significantly simplified for the case of convex functions.
 The arguments are presented in Appendix.  Consider  a log-concave measure $\nu = e^{-W} dx$
 such that
$$
W(x+y)+W(x-y) - W(x) \ge \delta(|y|)
$$
with a non-negative increasing function $\delta$.
Then the optimal transportation of the standard Gaussian measure to
$\nu$
 satisfies the following:
 $$
|\nabla \varphi(x) - \nabla \varphi(y)| \le
8 \delta^{-1}(4 |x-y|^2).
$$
This implies, in particular, that $\nu$ admits the following
 dimension-free concentration property:
$$
\nu\bigl( B_{r}  \bigr) \ge  1- \frac{1}{2} \exp \Bigl( -\frac{1}{8} \ \delta(r/8) \Bigr)
$$
for every $B \subset \R^d$ with
$\nu(B) \ge  1/2$, $B_r = \{y: |x-y| \le r, x \in B\}$.

The author expresses his gratitude to the colleagues from the Tel-Aviv university, especially
to Ronen Eldan and Sasha Sodin for communicating the proof of Lemma \ref{Sodin-lem} and to Emanuel Milman
for pointing out a mistake in the earlier version.

\section{Global H{\"o}lder estimates}

We deal throughout the paper with the standard finite-dimensional Euclidean space.

We recall that for every couple of probability measures
$\mu_1 = \rho_1 dx$, $\mu_2 = \rho_2 dx$ there exists
a mapping $T : \R^d \to \R^d$ pushing forward $\mu_1$
to $\mu_2$ which has the form $T = \nabla \varphi$, where $\varphi$ is a convex function
(see \cite{Vill} for details). For smooth $T$ one has the following
change of variables formula (Monge-Amp{\`e}re equation):
$$
\rho_2(\nabla \varphi) \det D^2 \varphi = \rho_1.
$$

All the measures considered below are supposed to have a convex support
of positive Lebegue measure.

Recall that a probability measure $\mu$ is called log-concave if
it has the form $\mu = e^{-V} dx$ with a convex $V$.
Finally, recall that a mapping $T : \R^d \to \R^d$ is $M$-Lipschitz if
$$\frac{|T(x)-T(y)|}{|x-y|} \le M.$$ By the Rademacher theorem $T$
is almost everywhere differentiable with $\| DT\| \le M.$
A mapping $T$ is $\alpha$-H{\"o}lder if there exists $0<\alpha \le 1$ such that
$$\frac{|T(x)-T(y)|}{|x-y|^{\alpha}} \le M.$$

We start with the one-dimensional case.
Let $T=\varphi'$ be the optimal transportation, pushing forward
$\mu = e^{-V} dx$ to $\nu = e^{-W} dx$ with $\mbox{supp}(\nu)=[a,b]$, where $V: \R \to \R$ and $W : (a,b) \to \R$
satisfy
\begin{equation}
\label{Vp}
|V'(x)-V'(y)| \le C_{p} |x-y|^{p},
\end{equation}
\begin{equation}
\label{Wq}
(x-y) (W'(x)-W'(y)) \ge C_{q} |x-y|^{q+1}
\end{equation}
for some $0< p \le 1$,  $1 \le q $.

Measures $\mu$ and $\nu$ are supposed to be
probability measures.
We are looking for a maximum point $(t_0,x_0)$ of the following
function
$$
F(t,x) = \frac{\varphi'(x+t)-\varphi'(x)}{|t|^{\alpha}}
\mbox{sign}(t), \ 0 < \alpha  \le 1.
$$

\begin{remark}
The arguments below are non-rigorous.
A more general version of the result (with a rigorous proof) see in Corollary \ref{hoelder-cor}.
\end{remark}

Setting $t'_0=-t_0$ and $x'=x_0-t_0$ if necessary, we may assume that
$t>0$. Differentiating in $x$ at the maximum point yields
$$
\varphi''(x_0+t_0)=\varphi''(x_0),
$$
$$
\varphi'''(x_0+t_0)\le \varphi'''(x_0).
$$
Hence
\begin{equation}
\label{1610}
\frac{\varphi'''(x_0+t_0)}{\varphi''(x_0+t_0)} -
\frac{\varphi'''(x_0)}{\varphi''(x_0)} \le 0.
\end{equation}
Taking derivative in $t$ we obtain
$$
\varphi''(x_0+t_0) = \frac{\alpha}{t_0} \Bigl( \varphi'(x_0+t_0) - \varphi'(x_0)\Bigr).
$$
Let us differentiate the change of variables formula.
One gets
$$
\frac{\varphi'''}{\varphi''}
- W'(\varphi')\varphi'' = -V'.
$$
Using (\ref{1610}) we obtain
\begin{align*}
V'(x_0+t_0) -V'(x_0)&
\ge
\bigl( W'(\varphi(x_0+t_0)) - W'(\varphi'(x_0)) \bigr)
\varphi''(x_0+t)
\\&
=
\frac{\alpha}{t}\Bigl[ W'(\varphi(x_0+t_0)) - W'(\varphi'(x_0)) \Bigr]
 \Bigl( \varphi'(x_0+t_0) - \varphi'(x_0)\Bigr).
\end{align*}
Then it follows from the assumptions on $V$ and $W$ that
$$
C_p \ t_0^{p+1} \ge \alpha C_q \ |\varphi'(x_0+t_0) - \varphi'(x_0)|^{q+1}.
$$
Hence
$$
\Bigl( \frac{C_p}{\alpha C_q} \Bigr)^{\frac{1}{q+1}} \  \ge  \ |\varphi'(x_0+t_0) - \varphi'(x_0)| \cdot t_0^{-\frac{p+1}{q+1}}.
$$
Thus we get the following statement:

 \begin{theorem}
 \label{1-hoelder}
 Assume that $V: \R \to \R$ and $W: \R \to \R$
are continuously differentiable functions
satisfying (\ref{Vp}) and (\ref{Wq}) for some
$0 \le p \le 1$, $q \ge 1$.
Set:
$$
0 < \alpha := \frac{p+1}{q+1} \le 1.
$$
Then
$$
|\varphi'(x+t)-\varphi'(x)| \le \Bigl( \frac{(q+1) C_p }{(p+1) C_q} \Bigr)^{\frac{1}{q+1}}
|t|^{\alpha}.
$$
\end{theorem}

\begin{remark}
A closed result has been established in \cite{BarKol}:
any  optimal mapping sending the one-dimensional exponential measure
$\nu = \frac{1}{2} e^{-|x|} dx $ to $e^{-V} \cdot \nu$
with $|V'| \le 1-c$, $c>0$ is $\frac{1}{c}$-Lipschitz.
\end{remark}

We assume that $V$ and $W$ satisfy
\begin{equation}
\label{vpm}
V(x+ y) + V(x-y) -2V(x) \le A_p |y|^{p+1},
\end{equation}
\begin{equation}
\label{wqm}
W(x+ y) + W(x-y) -2W(x) \ge A_q |y|^{q+1},
\end{equation}
for some $0 \le p \le 1$, $1 \le q$, $A_p>0$, $A_q>0$.

\begin{remark}
It can be easily verified that assumptions (\ref{vpm}), (\ref{wqm})
only make sense if $0 \le p \le 1$, $1 \le q$. Indeed, assuming other values of $p, q$,
one can easily deduce that the second derivatives of $W$, $V$ are either zero or infinity
everywhere.
\end{remark}

\begin{theorem}
\label{hoelder}
Let $\mu$ and $\nu$ satisfy  (\ref{vpm}), (\ref{wqm})
for some $0 \le p \le 1$, $1 \le q$. Then  $\varphi$
satisfies
\begin{equation}
\label{var-Hoelder}
\varphi(x+th) + \varphi(x-th)
-2\varphi(x) \le 2\Bigl( \frac{A_p}{A_q}\Bigr)^{\frac{1}{q+1}} t^{1+\alpha}
\end{equation}
for every unit vector $h \in \R^d$
with $\alpha = \frac{p+1}{q+1}$.
\end{theorem}
\begin{proof}
To prove the multi-dimensional case we follow
the arguments of Caffarelli from \cite{Caf}.
We consider the differential quotient
$$
\delta_2 \varphi(x) = \varphi(x+th) + \varphi(x-th)
-2\varphi(x) \ge 0
$$
for some vector $h \in \R^d$ with $|h|=1$.
Without loss of generality we may assume that
the support of $\nu$ is a bounded convex domain.
Note that according to a result from \cite{Caf92}
$\varphi$ is twice continuously differentiable.
It follows from the arguments of \cite{Caf} that
$\lim_{x\to \infty} \delta_2 \varphi(x) = 0$.
Thus there exists a maximum point $x_0$ of $\delta_2 \varphi(x)$.
Differentiating at $x_0$ yields
\begin{equation}
\label{maxim1}
\nabla \varphi(x_0+th) + \nabla \varphi(x_0-th) =2 \nabla \varphi(x_0),
\end{equation}
$$
D^2 \varphi(x_0+th) + D^2 \varphi(x_0-th) \le 2 D^2 \varphi(x_0).
$$
It follows from the concavity of determinant that
\begin{align*}
\det  D^2 \varphi(x_0)
&
\ge
\det \Bigl( \frac{D^2 \varphi(x_0+th) + D^2 \varphi(x_0-th)}{2} \Bigr)
\\&
\ge
\Bigl( \det D^2 \varphi(x_0+th) \ \det D^2 \varphi(x_0-th)\Bigr)^{\frac{1}{2}}.
\end{align*}
Applying the change of variables formula
$
\det D^2 \varphi = e^{W(\nabla \varphi) -V}
$ one finally gets
\begin{align}
\label{V-W}
V(x_0 + th) + V(x_0 - th) & - 2 V(x_0) \ge
\\&
\nonumber
W(\nabla \varphi(x_0 + th)) + W(\nabla \varphi(x_0 -th))
- 2 W(\nabla \varphi(x_0)).
\end{align}

It follows from  (\ref{maxim1}) that
$ v:=
\nabla \varphi(x_0+th) - \nabla \varphi(x_0) = \nabla \varphi(x_0) -  \nabla \varphi(x_0-th).
$
Hence we get by (\ref{V-W}) that
$$
A_p t^{p+1} \ge  A_q |\nabla \varphi(x_0+th) - \nabla \varphi(x_0)|^{q+1}
= A_q | \nabla \varphi(x_0-th) - \nabla \varphi(x_0)|^{q+1}
= A_q |v|^{q+1}.
$$
By convexity of $\varphi$
\begin{align*}
\varphi(x_0+th) + \varphi(x_0-th) - 2\varphi(x_0)
&
\le t \langle \nabla \varphi(x_0+th) - \nabla \varphi(x_0-th), h \rangle
\\&
= 2t \langle v,h \rangle \le 2t |v|.
\end{align*}
Finally
$$
A_p t^{p+1}
\ge A_q \Bigl( \frac{\delta_2 \varphi}{2t}\Bigr)^{q+1}.
$$
\end{proof}

\begin{remark}
For the case $p=q=1$ the estimate is known to be slightly better:
$$
\varphi(x+th) + \varphi(x-th)
-2\varphi(x) \le \Bigl( \frac{A_p}{A_q}\Bigr)^{\frac{1}{2}} t^{2}.
$$
\end{remark}

Our next goal is to  establish H{\"older} continuity of $\nabla \varphi$.

\begin{proposition}
\label{Stein-lem}
Let $u:\R^d \to \R$  be a differentiable function satisfying
$$
|u(x+ y) + u(x-y) -2u(x)| \le C |y|^{\alpha+1},
$$
$\alpha>0$ and $u, |\nabla u|$ be integrable with respect to every Gaussian measure.
Then there exists a constant $C'$ depending only on $\alpha$, $C$ and $d$
such that
$$
|u_v(x+y) - u_v(x)| \le C'|y|^{\alpha}
$$
for every unit vector $v \in \R^d$ and every $x,y \in \R^d$.
\end{proposition}

\begin{proof}
It is known that every bounded $f$ satisfying
$$|f(x+y) + f(x-y)
-2f(x)| \le C|y|^{\alpha+1}$$ admits a
H{\"o}lder continuous derivative: $|\nabla f(x+y) -\nabla f(x)| \le A|y|^{\alpha}$
(see \cite{Stein}, Chapter 5(4), Proposition 9).
Below we give a modification of the proof from \cite{Stein}.

Let us consider the heat semigroup  acting on $u$.
$$
P_t u = \int_{\R^d} u(x-y) p_t(y) dy, \ \ p_t(y)= \frac{1}{(2\pi t)^{d/2}} e^{-\frac{y^2}{2t}}.
$$
We fix an orthonormal basis $\{e_i\}$.
Clearly, it is sufficient to prove the statement for
$$|u_{x_i}(x+h e_j) -  u_{x_i}(x)|$$ with any $1 \le i,j \le d$ and some fixed $h>0$.
Note that
$|u_{x_i}(x+h e_j) -  u_{x_i}(x)|$ does not exceed
$$
|P_t u_{x_i}(x+h e_j) -  P_t u_{x_i}(x)|
+ |u_{x_i}(x+h e_j) -  P_t u_{x_i}(x+he_j)|
+ |u_{x_i}(x) -  P_t u_{x_i}(x)|,
$$
for every $t>0$.

The first term is estimated by
\begin{align*}
P_t & u_{x_i}(x+ h e_j) -  P_t u_{x_i}(x)
=
\int_{0}^{h} \int_{\R^d} u_{x_i x_j}(x+r e_j -y) p_t(y) \ dy \ dr
\\&
= \int_{0}^{h} \int_{\R^d} u(x+r e_j -y) p_t(y)_{y_i y_j} \ dy \ dr
\\&
= \frac{1}{2} \int_{0}^{h} \int_{\R^d} \bigl(u(x+r e_j -y) +
u(x+r e_j +y) - u(x+r e_j)\bigr)  p_t(y)_{y_i y_j} \ dy \ dr
\\&
\le
\frac{1}{2} \int_{0}^{h} \int_{\R^d} |u(x+r e_j -y) +
u(x+r e_j +y) -u(x+r e_j)| | p_t(y)_{y_i y_j}| \ dy \ dr.
\end{align*}

Obviously
$$
p_t(y)_{y_i y_j}
= \frac{1}{t} \Bigl( \frac{y_i y_j}{t} - \delta_{ij}\Bigr) p_t(y).
$$
Hence
$$
|P_t  u_{x_i}(x+ h e_j) -  P_t u_{x_i}(x)|
\le
\frac{C}{2t} \int_{0}^{h}
\int_{\R^d} |y|^{1 +\alpha}   \Bigl| \frac{y_i y_j}{t} - \delta_{ij}\Bigr| p_t(y) \ dy dr.
$$
It is clear by scaling arguments that the latter does not exceed
$$
C_1(C,\alpha,d) h \cdot t^{\frac{\alpha-1}{2}}.
$$
Thus we get for $t=h^2$:
$$
|P_t  u_{x_i}(x+ h e_j) -  P_t u_{x_i}(x)|
\le
C_1(C,\alpha,d) h^{\alpha}.
$$

Let us estimate the remaining terms.
To this end we consider
$\bigl( P_{t} u\bigr)_t$.
One gets
\begin{align*}
&
\bigl( P_{t} u\bigr)_t (x)
= \int_{\R^d} u(x-y) \Bigl[ \frac{1}{(2\pi t)^{d/2}} e^{-\frac{y^2}{2t}} \Bigr]_{t}dy
\\&
= \int_{\R^d} u(x-y) \Bigl( \frac{y^2}{2t^2} -\frac{d}{2t} \Bigr)p_t(y) dy
\\&
=
\frac{1}{2}
\int_{\R^d} \Bigl( u(x-y)  + u(x+y) - 2u(x) \Bigr)\Bigl( \frac{y^2}{2t^2} -\frac{d}{2t} \Bigr)p_t(y) dy
.
\end{align*}
Hence
\begin{equation}
\label{ptt}
|\bigl( P_{t} u\bigr)_t (x)|
\le
\frac{1}{2}
\int_{\R^d} |y|^{\alpha+1} \Bigl| \frac{y^2}{2t^2} -\frac{d}{2t} \Bigr| p_t(y) dy
\le C_2(C,\alpha,d)  \cdot t^{\frac{\alpha-1}{2}}.
\end{equation}
Integration yields
\begin{equation}
\label{pt}
| P_{t} u(x) - u(x)|
\le C_3(C,\alpha,d)  \cdot t^{\frac{\alpha+1}{2}}.
\end{equation}

Note  that
$$
u_{x_i}(x) - P_t u_{x_i}(x)
= - \int_{0}^{t} \bigl( P_s u \bigr)_{s x_i} ds,
$$
$$
P_s u = P_{s/2} u * p_{s/2}.
$$
Differentiating the convolution identity yields
\begin{align*}
P_t u_{x_i}(x) - u_{x_i}(x)
&
=
\int_{0}^{t} \Bigl( \int_{\R^d} \bigl( P_{s/2} u(y)\bigr)_s \bigl(p_{s/2}(x-y))_{x_i} dy \Bigr) ds
\\&
+ \int_{0}^{t}  \Bigl(\int_{\R^d}P_{s/2} u(y) \bigl(p_{s/2}(x-y))_{x_i s} dy \Bigr) ds.
\end{align*}
It is easy to check that
\begin{align*}
& \int_{0}^{t}  \Bigl(\int_{\R^d} P_{s/2}u(y) \bigl(p_{s/2}(x-y))_{x_i s} dy \Bigr) ds
=
\\&
=\int_{0}^{t}  \Bigl(\int_{\R^d} \bigl(P_{s/2}u-u \bigr)(y) \bigl(p_{s/2}(x-y))_{x_i s} dy \Bigr) ds
+  P_{t/2} u_{x_i}(x) - u_{x_i}(x).
\end{align*}
One has for $t=h^2$
\begin{align*}
P_t u_{x_i}(x) - P_{t/2} u_{x_i}(x)
&
=
\int_{0}^{t} \Bigl( \int_{\R^d} \bigl( P_{s/2} u\bigr)_s(y) \bigl(p_{s/2}(x-y))_{x_i} dy \Bigr) ds
\\&
+ \int_{0}^{t}  \Bigl(\int_{\R^d} \bigl(P_{s/2}u-u \bigr)(y) \bigl(p_{s/2}(x-y))_{x_i s} dy \Bigr) ds.
\end{align*}
Then it follows from (\ref{pt}), (\ref{ptt}) that
\begin{align*}
&
|P_t u_{x_i}(x) - P_{t/2} u_{x_i}(x)|
\le
\\&
C_4
\Bigl(
\int_{0}^{t} s^{\frac{\alpha-1}{2}} \int_{\R^d}  \bigl|p_{s/2}(x-y))_{x_i} \bigr| dy  ds
+ \int_{0}^{t} s^{\frac{\alpha+1}{2}}  \int_{\R^d} \bigl|p_{s/2}(x-y))_{x_i s}\bigr| dy  ds
\Bigr).
\end{align*}
Applying the same arguments as above we obtain
$$
|P_t u_{x_i}(x) - P_{t/2} u_{x_i}(x)|
\le
C_5
\int_{0}^{t} s^{\frac{\alpha-2}{2}} ds = C_6 t^{\frac{\alpha}{2}}.
$$
Hence
\begin{align*}
& |P_t u_{x_i}(x) - u_{x_i}(x)|
\\&
\le
\sum_{k=0}^{\infty} |P_{t/2^k} u_{x_i}(x) - P_{t/2^{k+1}} u_{x_i}(x)|
 \le C_6 \sum_{k=0}^{\infty} \Bigl( \frac{t}{2^k} \Bigr) ^{\frac{\alpha}{2}}
=C_7 t^{\alpha/2} = C_7 h^{\alpha}.
\end{align*}

The same estimate holds for the remaining third term.
The proof is complete.
\end{proof}

\begin{remark}
The statement does not hold for $\alpha=0$ (see \cite{Stein} 4.3.1).
\end{remark}

\begin{corollary}
\label{hoelder-cor}
Under assumptions of Theorem \ref{hoelder}
the optimal transportation $\nabla \varphi$
is H{\"o}lder continuous:
$$
|\nabla \varphi(x) -\nabla \varphi(y)| \le C |x-y|^{\alpha}
$$
with $C$ depending only on $p,q, A_p, A_q$.
\end{corollary}
\begin{proof}
Replace $\nu$ by $\frac{1}{\nu(A)} \nu|_{A}$ with a bounded convex $A$. Clearly, the new potential  satisfies
 assumptions of the theorem.
Hence, using standard approximation arguments we can restrict ourselves
to the case of compactly supported $\nu$. Thus we assume without loss of generality that
$|\nabla \varphi| \le K$.
By Theorem \ref{hoelder}
\begin{equation}
\label{24.10.08}
\varphi(x+th) + \varphi(x-th)
-2\varphi(x) \le C|t|^{\alpha}
\end{equation}
for every unit vector  $h$.
We get immediately from Proposition \ref{Stein-lem} that $\nabla \varphi$ is H{\"o}lder.
To see that the constant does not depend on dimension, let as fix $x,h  \in \R^d$ and $v \in \R^d$
with $|h|=|v| = 1$. Consider the restriction of $\varphi$ onto
the affine hyperplane $A_{x,v,h} = \{x + th +sv, \ s \in \R, t \in \R \}$.
This restriction clearly satisfies (\ref{24.10.08}).
Thus the $2$-dimensional version of Lemma \ref{Stein-lem} implies that
$$
|\langle \nabla \varphi(x+th)-\nabla \varphi(x), v \rangle|
= |\varphi_v(x+th) - \varphi_v(x)| \le C'|t|^{\alpha},$$
where $C'$ (a corresponding "two-dimensional" constant) does not depend on $d$ and on directions of $h,v$.
Since this can be repeated for every $x,h,v$ with  the same constant $C'$, one has
$$
|\nabla \varphi(x+th)-\nabla \varphi(x) |
=
\sup_{|v|=1}
|\langle \nabla \varphi(x+th)-\nabla \varphi(x), v \rangle|
\le
C'|t|^{\alpha}.
$$
The proof is complete.
\end{proof}

\begin{remark}
A more general statement see in Corolary \ref{MS-conc}.
\end{remark}

\section{Estimates for the second order derivatives: one-dimensional case}

We start with some  heuristic estimates in the one-dimensional case.
Consider a convex function $\varphi$ such that $\varphi'$ sends
$\mu = \frac{1}{\sqrt{2\pi}} e^{-\frac{x^2}{2}} \ dx$ to a probability measure $e^{-W} dx$.
We assume that $W$ is sufficiently smooth.

By the change of variables formula
\begin{equation}
\label{cvf}
-\frac{\log 2 \pi}{2} - \frac{x^2}{2} = -W(\varphi') + \log \varphi''.
\end{equation}
Assume that $f$ is a smooth function such that
$
\frac{\varphi''}{f(\varphi')}
$
admits its maximum at some point $x_0$.
One has at this point
\begin{equation}
\label{1'}
\frac{\varphi^{(3)}}{f(\varphi')} - \frac{(\varphi'')^2 f'(\varphi') }{f^2(\varphi')} = 0,
\ \ \
\frac{\varphi^{(4)}}{f(\varphi')} - \frac{3 \ \varphi^{(3)} \varphi'' f'(\varphi') }{f^2(\varphi')}
- (\varphi'')^3 \Bigl[ \frac{f'}{f^2}\Bigr]' (\varphi')
 \le 0.
\end{equation}
Thus we get
\begin{equation}
\label{3'}
\varphi^{(3)}= \frac{(\varphi'')^2 f'(\varphi') }{f(\varphi')},
\end{equation}
Applying (\ref{3'}) we find
\begin{equation}
\label{4'}
\varphi^{(4)} \le (\varphi^{''})^3 \Bigl[ \frac{f''}{f} + \Bigl( \frac{f'}{f} \Bigr)^2\Bigr](\varphi').
\end{equation}
Differentiating (\ref{cvf}) twice yields
$$
0 = 1 - W''(\varphi') (\varphi^{''})^2 - W'(\varphi') \varphi^{(3)}
+ \frac{\varphi^{(4)}}{\varphi''} - \Bigl( \frac{\varphi^{(3)}}{\varphi''}\Bigr)^2.
$$
Applying (\ref{3'}) and (\ref{4'}) we get
$$
 -(\varphi^{''})^2 \frac{f''(\varphi')}{f(\varphi')} + W''(\varphi') (\varphi^{''})^2 + W'(\varphi') \frac{(\varphi'')^2 f'(\varphi') }{f(\varphi')}
  \le 1.
$$
Now assume that $f$ satisfies
$
 - \frac{f''}{f} + W''+ W' \frac{ f'}{f} \ge \frac{1}{f^2}.
$
Thus we get  at  $x_0$
$$
\Bigl[ \frac{\varphi''}{f(\varphi')} \Bigr]^2 \le 1.
$$

Since $x_0$ is supposed to be the point of maximum
for $\frac{\varphi''}{f(\varphi')}$, we get a formal proof
of the following statement:

{\it Assume that $f$ and $W$ satisfy
$$
 - \frac{f''}{f} + W''+ W' \frac{ f'}{f} \ge \frac{1}{f^2}.
$$
Then
$$
\varphi'' \le f(\varphi').
$$
}

Of course, these arguments are non-rigorous. Nevertheless,
they are applicable for many reasonable situations.
\begin{example}
For the case $W'' \ge K$ and $f \equiv \frac{1}{\sqrt{K}}$ we get a refinement of the  Caffarelli's result:
$$
\varphi'' \le \frac{1}{\sqrt{K}}.
$$
\end{example}

\begin{example}
\label{1dim}
Assume that
$$
e^{-W} = \frac{1}{2a} I_{[-a,a]}.
$$
Applying the arguments from above, we get the following estimate:
$$
\varphi'' \le f(\varphi'),
$$
where $f:[-a,a] \to \R^+$ is a concave even function, decreasing of $[0,a]$,
such that $f(a)=f(-a)=0$ and
$$
-ff^{''}=1.
$$
This function can be found explicitly:
$$
f(t) =f(0) \cdot \Phi^{-1} \Bigl( \frac{\sqrt{\frac{\pi}{2}}f(0) -t}{f(0)} \Bigr).
$$
Here
$
\Phi: [0,1] \to \Bigl[0, \sqrt{\frac{\pi}{2}}\Bigr]
$
is defined by
$$
\Phi(x) = \int_{0}^{x} \frac{ds}{\sqrt{-2 \ln s}}
= \int_{\sqrt{-2 \ln x}}^{\infty} e^{-\frac{s^2}{2}} ds.
$$
In addition, $a$ and $f(0)$ are related by
$$
f(0) = a \sqrt{\frac{2}{\pi}}.
$$
Thus
$$
\varphi'' \le a \sqrt{\frac{2}{\pi}} \Phi^{-1} \Bigl( \sqrt{\frac{\pi}{2}} \cdot \frac{a -\varphi'}{a} \Bigr).
$$

Note that this example can be obtained directly from the Caffarelli theorem. Indeed, set:
$$
G(x) = \int_{0}^{x} \frac{ds}{f(s)} = -f'(x).
$$
It can be verified by the direct computation that
$G$ pushes forward $\frac{1}{2a} I_{[-a,a]}$ to $\nu = \frac{1}{\gamma([-b,b])} \cdot \gamma|_{[-b,b]}$,
where $b = G(a)$. Note that $T = G(\varphi')$ is the optimal transportation
of $\gamma$ to $\nu$. Hence, by the Caffarelli theorem, $T$ is $1$-Lipschitz:
 $T' \le 1$. Thus $\frac{\varphi''}{f(\varphi')} \le 1$.
\end{example}

\section{Estimates for the second order derivatives: multidimensional case}

For the rest of the paper $\nu$ is the uniform measure
on a bounded convex domain  $K$. The potential $V$ is supposed to be at least two times
 differentiable with a H{\"o}lder second derivative.

Note that  the assumption of convexity of $K$ is crucial for smoothness of $\varphi$.
It is well known that for non-convex sets the potential $\varphi$ is not smooth in general.
Nevertheless, in our case the Caffarelli's regularity theory ensures that
$\varphi$ is smooth (see \cite{Caf92}, \cite{Caf96}).
Indeed, let us take any convex domain $K' \subset K$ such that
$\partial K'$ lies positive distance from $K$. Let $K_0 = T^{-1} K'.$
It is easy to check that $K_0$ is bounded (otherwise we get the contradiction with the monotonicity of $T$).
Then applying a result from \cite{Caf92} we immediately get that $T$ is differentiable with a
H{\"o}lder continuous derivative
inside of $K_0$. The further regularity follows from the smoothness of $V$ by the classical arguments
(see \cite{Caf92}, Remark 4.15 in \cite{Vill}).

\begin{definition}
For a positive $a$ and  $p \in \R$ let
$f(t):=f_{p,a}(t) : [-a,a] \to \R^{+}$ be a function satisfying:
\begin{itemize}
\item[1)]
$
ff'' + p(f')^2=-1
$
\item[2)]
$f$ is even and $f(a)=0$.
\end{itemize}
\end{definition}

 Assumptions 1)-2)  define uniquely a  function $f$, which
decreases on $[0,a]$.
Assume for a while that $p>0$.
First, taking $f$ as a new variable and integrating 1), one gets
$$
f' = - \frac{1}{\sqrt{p}} \sqrt{\Bigl( \frac{f(0)}{f}\Bigr)^{2p}-1}.
$$

This yields the following expression for $f$:
$$
f(t) = f(0) \sqrt[p]{\Psi^{-1} \Bigl( \sqrt{p} \ \frac{t}{f(0)} \Bigr)},
$$
where
$
\Psi : [0,1] \to [0,\int_{0}^{\frac{\pi}{2}} \sin^{\frac{1}{p}}x \ dx],
$
$$
\Psi(t) = \int_{t}^1 \frac{r^{\frac{1}{p}} dr}{\sqrt{1-r^2}}.
$$
In addition, $a$ and $f(0)$ are related by
$$
a \sqrt{p} = f(0) \int_{0}^1 \frac{r^{\frac{1}{p}} dr}{\sqrt{1-r^2}}
=
f(0) \int_{0}^{\frac{\pi}{2}} \sin^{\frac{1}{p}}x \ dx
.
$$

The case $p=0$ has been considered in Example \ref{1dim}.
The case $p<0$ is similar to $p>0$. One has
$$
f(t) = f(0) \sqrt[p]{\Psi^{-1} \Bigl( \sqrt{-p} \ \frac{t}{f(0)} \Bigr)},
$$
where
$
\Psi : [1,\infty] \to [0,\int_{0}^{\frac{\pi}{2}} \sin^{-1-\frac{1}{p}}x \ dx],
$
$$
\Psi(t) = \int_{1}^t \frac{r^{\frac{1}{p}} dr}{\sqrt{r^2-1}}.
$$
In addition, $a$ and $f(0)$ are related by
$$
a \sqrt{-p} = f(0) \int_{1}^{\infty} \frac{r^{\frac{1}{p}} dr}{\sqrt{r^2-1}}
=
f(0) \int_{0}^{\frac{\pi}{2}} \sin^{-1-\frac{1}{p}}x \ dx
.
$$

\begin{theorem}
Let $K \subset \R^d$, $d>1$ be a bounded convex set of a positive volume.
Let $\nabla \varphi : \R^d \to K$ be the optimal transportation pushing forward
probability measure $\mu = e^{-V} dx$ with smooth $V$
to the uniform measure $\frac{1}{\lambda(K)} \lambda|_K$. Assume that $V_{hh} \le \Lambda$ for some $h \in \R^d$
with $|h|=1$. Then one has
$$
\varphi_{hh} \le \sqrt{\Lambda} \cdot f_{\frac{d-1}{4}, a} (\varphi_h - t_0),
$$
where $t_0$ is chosen in such a way that
$$L_1 = \{x: \langle x,h \rangle = t_0-a\}, \ L_2 = \{x: \langle x,h \rangle = t_0+a\}$$
are supporting hyperplanes to $K$.

If, in addition, we assume that
$|V_h| \le M$ then the following dimension-free estimate holds:
$$
\varphi_{hh} \le \sqrt{\Lambda + \frac{M^2}{4(1+p)}} \cdot f_{p, a} (\varphi_h - t_0)
$$
for any $p>-1$.

\end{theorem}
\begin{proof}
Note that the estimates are invariant with respect to any shift of the space.
Let us shift $K$ is such a way that $K$ contains the origin.
This clearly implies that $\lim_{x \to \infty} \varphi(x)=+\infty$.
We are looking for a maximum of
$
\varphi_{hh}(x) e^{\psi(\varphi_h(x))}
$
among  all of $x \in \R^d$, with $\psi$ to be chosen later.
To apply the maximum principle and make sure that the maximum is attained
we deal with the following
compactly supported modification:
$$
F_{\varepsilon}(x) =
(1-\varepsilon \varphi)_{+} \varphi_{hh}(x) e^{\psi(\varphi_h(x))}.
$$

Consider the change of variables formula for $\varphi$
$$
C(K) + \log \det D^2 \varphi = -V.
$$
Let $y$ be a desired maximum point of $F_{\varepsilon}$. All the computations below are made at this point.
First we change the coordinate system linearly in such a way that
$$
1) h=e_1, ~ \mbox{2) $D^2 \varphi(y)$ is diagonal}.
$$
The first  requirement is achieved just by a  rotation. In addition, without loss of generality we may assume that
$(D^2\varphi)_{i,j}$, $i,j \ge 2$ is diagonal.
 To fulfill the second one we choose
a non-orthogonal linear transformation.
$$
x'_1 = x_1 + \frac{\varphi_{x_2 x_1}(y)}{\varphi_{x_1 x_1}(y)} x_2 + \cdots
+ \frac{\varphi_{x_d x_1}(y)}{\varphi_{x_1 x_1}(y)} x_d,
$$
$$
x'_{i} = x_{i}, ~ i \in \{2,\cdots,d\}.
$$
Here $(x'_1, \cdots, x'_d)$ is the "new" coordinate system and $(x_1, \cdots, x_d)$ is the "old" one.
One checks easily  that $\varphi_{x'_1 x'_i}=0$ for $i \ge 2$.
Since the transformation is represented by a unimodal matrix, the change of variables formula
looks in new coordinates as follows:
\begin{equation}
\label{MAE+}
C(K) + \log \det D^{2}_{x'} \varphi = - V\Bigl(x'_1 - \frac{\varphi_{x_2 x_1}(y)}{\varphi_{x_1 x_1}(y)} x'_2 - \cdots
- \frac{\varphi_{x_d x_1}(y)}{\varphi_{x_1 x_1}(y)} x'_d, x'_2, \cdots, x'_d\Bigr),
\end{equation}
where  $D^{2}_{x'} \varphi$ denotes the matrix of second derivatives computed in new coordinate system.
In what follows we set:
$x'_1 - \frac{\varphi_{x_2 x_1}(y)}{\varphi_{x_1 x_1}(y)} x'_2 - \cdots
- \frac{\varphi_{x_d x_1}(y)}{\varphi_{x_1 x_1}(y)} x'_d = Ox'$
and write
for simplicity  $x_i$ instead of $x'_i$.

Let us differentiate ate (\ref{MAE+}). To this end it is convenient to use the
following expression for the determinant
$$
\det D^2\varphi = \sum_{\sigma \in S(d)} \mbox{sign} (\sigma)
\varphi_{x_1 x_{\sigma(1)}} \cdots \varphi_{x_d x_{\sigma(d)}}
$$
 and taking into account that $\varphi_{x_i x_j}(y) =0$ if $i \ne
j$, we get at $y$:

\begin{equation}
\label{1det}
-V_{x_1}(Ox) =\frac{(\det D^2 \varphi)_{x_1}}{\det D^2 \varphi} = \sum_{i}
\frac{\varphi_{x_i x_i x_1}}{\varphi_{x_i x_i}} ,
\end{equation}

\begin{align}
\label{2det}
-V_{x_1 x_1}(Ox)  & =
\frac{(\det D^2 \varphi)_{x_1 x_1}}{\det D^2 \varphi}
- \Bigl( \frac{(\det D^2 \varphi)_{x_1}}{\det D^2 \varphi}\Bigr)^2 =
\\& \nonumber = \sum_{i}
\frac{\varphi_{x_i x_i x_1 x_1}}{\varphi_{x_i x_i}} + \sum_{i \ne j}
\frac{\varphi_{x_i x_i x_1}}{\varphi_{x_i x_i}} \frac{\varphi_{x_j x_j
x_1}}{\varphi_{x_j x_j}}
- \sum_{i \ne j} \frac{\varphi^2_{x_i x_j x_1}}{\varphi_{x_i x_i} \varphi_{x_j x_j}}.
\end{align}

Differentiating (\ref{MAE+}) twice in $x_1$ yields
$$
-\Lambda \le -V_{x_1x_1}(Ox) = \Bigl[\log \det D^2 \varphi \Bigr]_{x_1 x_1}
=
\frac{(\det D^2 \varphi)_{x_1 x_1}}{\det D^2 \varphi}
-
\Bigl(\frac{(\det D^2 \varphi)_{x_1}}{\det D^2 \varphi}\Bigr)^2.
$$

It follows from (\ref{1det}), (\ref{2det}) that
\begin{equation}
\label{9.10}
-\Lambda \le \sum_{i}
\frac{\varphi_{x_i x_i x_1 x_1}}{\varphi_{x_i x_i}}
- \sum_{i, j} \frac{\varphi^2_{x_i x_j x_1}}{\varphi_{x_i x_i} \varphi_{x_j x_j}}.
\end{equation}

Since $\log F_{\varepsilon}$ admits its maximum at $y$, one has
at this point
\begin{equation}
\label{x1}
-\frac{\varepsilon}{1-\varepsilon \varphi} \varphi_{x_1}
+
\frac{\varphi_{x_1 x_1 x_1}}{\varphi_{x_1 x_1}}
+ \psi'(\varphi_{x_1}) \varphi_{x_1 x_1} =0,
\end{equation}
\begin{equation}
\label{xi}
-\frac{\varepsilon}{1-\varepsilon \varphi} \varphi_{x_i}
+
\frac{\varphi_{x_1 x_1 x_i}}{\varphi_{x_1 x_1}} =0, ~ i \ge 1.
\end{equation}
Differentiating $\log F_{\varepsilon}$ twice in $x_1$ yields
\begin{align*}
-\frac{\varepsilon}{1-\varepsilon \varphi} \varphi_{x_i x_i}
&
- \frac{\varepsilon^2}{\bigl(1-\varepsilon \varphi\bigr)^2} \varphi^2_{x_i}
+ \frac{\varphi_{x_1 x_1 x_i x_i}}{\varphi_{x_1 x_1}}
- \Bigl( \frac{\varphi_{x_1 x_1 x_i}}{\varphi_{x_1 x_1}} \Bigr)^2
\\&
+ \psi^{''}(\varphi_{x_1}) \varphi^2_{x_1 x_i}
+\psi^{'}(\varphi_{x_1}) \varphi_{x_1 x_i x_i} \le 0.
\end{align*}
Multiplying this inequality on $\frac{\varphi_{x_1 x_1}}{\varphi_{x_i x_i}}$,
 summing in $i$ and applying (\ref{9.10}) one gets
\begin{align}
\label{02.10.08}
&
-\frac{d\varepsilon}{1-\varepsilon \varphi} \varphi_{x_1 x_1}
- \frac{\varepsilon^2}{\bigl(1-\varepsilon \varphi\bigr)^2} \varphi_{x_1 x_1}
\Bigl(\frac{\varphi^2_{x_1}}{\varphi_{x_1 x_1}} + \cdots +\frac{\varphi^2_{x_d}}{\varphi_{x_d x_d}} \Bigr)
\\&
\nonumber
+ \sum_{i, j} \frac{\varphi^2_{x_i x_j x_1}}{\varphi_{x_i x_i} \varphi_{x_j x_j}}
- \sum_{i} \frac{\varphi^2_{x_1 x_1 x_i}}{\varphi_{x_1 x_1} \varphi_{x_i x_i}}
\\&
\nonumber
+\psi^{''}(\varphi_{x_1}) \varphi^2_{x_1 x_1}
+
\psi^{'}(\varphi_{x_1})
\Bigl[ \sum_{i} \frac{\varphi_{x_1 x_i x_i}}{\varphi_{x_i x_i}} \Bigr] \varphi_{x_1 x_1}
\le \Lambda.
\end{align}
Note that
\begin{align*}
&
\sum_{i, j} \frac{\varphi^2_{x_i x_j x_1}}{\varphi_{x_i x_i} \varphi_{x_j x_j}}
- \sum_{i} \frac{\varphi^2_{x_1 x_1 x_i}}{\varphi_{x_1 x_1} \varphi_{x_i x_i}}
= \sum_{i \in \{2,\cdots,d\}}
\sum_{j \in \{1,\cdots,d\}} \frac{\varphi^2_{x_i x_j x_1}}{\varphi_{x_i x_i} \varphi_{x_j x_j}}
\\&
\ge
\sum_{j \in \{2,\cdots,d\}}  \frac{\varphi^2_{x_j x_j x_1}}{\varphi^2_{x_j x_j}}
+
\sum_{i \in \{2,\cdots,d\}}  \frac{\varphi^2_{x_i x_1 x_1}}{\varphi_{x_i x_i} \varphi_{x_1 x_1}}.
\end{align*}
One gets by (\ref{xi})
$$
\sum_{i, j} \frac{\varphi^2_{x_i x_j x_1}}{\varphi_{x_i x_i} \varphi_{x_j x_j}}
- \sum_{i} \frac{\varphi^2_{x_1 x_1 x_i}}{\varphi_{x_1 x_1} \varphi_{x_i x_i}}
\ge
\sum_{j \in \{2,\cdots,d\}}  \frac{\varphi^2_{x_j x_j x_1}}{\varphi^2_{x_j x_j}}
+
\sum_{i \in \{2,\cdots,d\}}  \frac{\varepsilon^2}{(1-\varepsilon \varphi)^2}
\frac{\varphi_{x_1 x_1}}{\varphi_{x_i x_i}} \varphi^2_{x_i}.
$$
Hence it follows from (\ref{02.10.08})
\begin{align*}
&
-\frac{d\varepsilon}{1-\varepsilon \varphi} \varphi_{x_1 x_1}
- \frac{\varepsilon^2}{\bigl(1-\varepsilon \varphi\bigr)^2}
\varphi^2_{x_1}
+
\sum_{j \in \{2,\cdots,d\}}  \frac{\varphi^2_{x_j x_j x_1}}{\varphi^2_{x_j x_j}}
\\&
+\psi^{''}(\varphi_{x_1}) \varphi^2_{x_1 x_1}
+
\psi^{'}(\varphi_{x_1})
\Bigl[ \sum_{i \in \{2,\cdots d\}} \frac{\varphi_{x_1 x_i x_i}}{\varphi_{x_i x_i}} \Bigr] \varphi_{x_1 x_1}
+
\psi^{'}(\varphi_{x_1}) \varphi_{x_1 x_1 x_1}
\le \Lambda.
\end{align*}
One obtains from (\ref{x1})
$$
\varphi_{x_1 x_1 x_1}
 =\frac{\varepsilon}{1-\varepsilon \varphi} \varphi_{x_1} \varphi_{x_1 x_1}
 - \psi'(\varphi_{x_1}) \varphi^2_{x_1 x_1}.
$$
By the Cauchy inequality
$$
\sum_{j \in \{2,\cdots,d\}}  \frac{\varphi^2_{x_j x_j x_1}}{\varphi^2_{x_j x_j}}
+
\psi^{'}(\varphi_{x_1})
\Bigl[ \sum_{i \in \{2,\cdots d\}} \frac{\varphi_{x_1 x_i x_i}}{\varphi_{x_i x_i}} \Bigr] \varphi_{x_1 x_1}
\ge
-\frac{(d-1)}{4}
(\psi^{'}(\varphi_{x_1}))^2 \varphi^2_{x_1 x_1}.
$$
Hence
\begin{align*}
&
-\frac{d\varepsilon}{1-\varepsilon \varphi} \varphi_{x_1 x_1}
- \frac{\varepsilon^2}{\bigl(1-\varepsilon \varphi\bigr)^2}
\varphi^2_{x_1}
+\frac{\varepsilon}{1-\varepsilon \varphi} \psi^{'}(\varphi_{x_1}) \varphi_{x_1} \varphi_{x_1 x_1}
\\&
+\bigl[ \psi^{''}(\varphi_{x_1}) - \frac{d+3}{4}(\psi^{'}(\varphi_{x_1}))^2 \bigr]\varphi^2_{x_1 x_1}
\le \Lambda.
\end{align*}
Applying the Cauchy inequality one gets
$$
\Bigl[\psi^{''} - \Bigl(\frac{d+3}{4}+\varepsilon\Bigr) (\psi')^2 - d^2\varepsilon^2 \Bigr]\circ \varphi_{x_1}
\cdot \varphi^2_{x_1 x_1} (1-\varepsilon \varphi)^2
\le
\Lambda (1-\varepsilon \varphi)^2
+ (2\varepsilon^2 \varphi^2_{x_1} + \varepsilon).
$$
Let $\psi_{\varepsilon}$ be a smooth function on
$[\inf_{y \in K} \langle x_1, y \rangle, \sup_{y \in K} \langle x_1, y \rangle  ]$
satisfying
\begin{equation}
\label{28.10.08}
\psi^{''}_{\varepsilon} \ge  \bigl(\frac{d+3}{4}+\varepsilon\bigr) (\psi'_{\varepsilon})^2 + d^2\varepsilon^2 +
e^{2\psi_{\varepsilon}}.
\end{equation}
The maximum principle implies
$$
e^{\psi_{\varepsilon}(\varphi_{x_1})} \varphi_{x_1 x_1} (1-\varepsilon \varphi)_{+}
\le
\max_{x \in \R^d} \bigl( e^{\psi_{\varepsilon}(\varphi_{x_1})} \varphi_{x_1 x_1} (1-\varepsilon \varphi)_{+} \bigr)
\le \sqrt{ \Lambda  + \varepsilon +  2\varepsilon^2 \sup_{y \in K} \langle y, x_1 \rangle^2}.
$$
Set: $f_{\varepsilon} := e^{-\psi_{\varepsilon}}$. The differential inequality for $\psi_{\varepsilon}$
can be rewritten in the following way:
\begin{equation}
\label{varep-di}
-f_{\varepsilon} f^{''}_{\varepsilon}
\ge
\Bigl(\frac{d-1}{4} + \varepsilon\Bigr) (f'_{\varepsilon})^2 + 1 +
 \varepsilon^2 d^2 f^2_{\varepsilon}.
\end{equation}

The following easy-to-check observation completes the proof:
there exists a sequence of nonnegative functions $\{f_{\varepsilon}\}$ on
$[-a,a]$ satisfying
(\ref{varep-di}) such that $f_{\varepsilon} \to f_{\frac{d-1}{4}, a}$ uniformly on
$[-a,a ]$.
In the limit $\varepsilon \to 0$ we get from (\ref{28.10.08})
$$\varphi_{x_1 x_1} \le \sqrt{\Lambda} f_{\frac{d-1}{4}, a}(\varphi_{x_1}).$$
The proof of the first statement is complete.

To prove the second statement we use (\ref{1det}) and (\ref{02.10.08}) to get
\begin{align*}
&
-\frac{d\varepsilon}{1-\varepsilon \varphi} \varphi_{x_1 x_1}
- \frac{\varepsilon^2}{\bigl(1-\varepsilon \varphi\bigr)^2} \varphi_{x_1 x_1}
\Bigl(\frac{\varphi^2_{x_1}}{\varphi_{x_1 x_1}} + \cdots +\frac{\varphi^2_{x_d}}{\varphi_{x_d x_d}} \Bigr)
+ \sum_{i, j} \frac{\varphi^2_{x_i x_j x_1}}{\varphi_{x_i x_i} \varphi_{x_j x_j}}
\\&
- \sum_{i} \frac{\varphi^2_{x_1 x_1 x_i}}{\varphi_{x_1 x_1} \varphi_{x_i x_i}}
+\psi^{''}(\varphi_{x_1}) \varphi^2_{x_1 x_1}
+
\psi^{'}(\varphi_{x_1})
V_{x_1}(Ox) \varphi_{x_1 x_1}
\le \Lambda.
\end{align*}
Since $|V_{x_i}|\le M$, arguing as above and passing to the limit $\varepsilon \to 0$,
one gets the desired inequality for $\psi$:
$$
\psi^{''}(\varphi_{x_1}) \varphi^2_{x_1 x_1}
-M |\psi^{'}(\varphi_{x_1})| \varphi_{x_1 x_1}
\le \Lambda.
$$
Setting $f = -\log \psi$ and applying  Cauchy inequality we obtain that for $f$
satisfying
$$
-ff^{''} + \Bigl(1-\frac{\delta}{2} M^2 \Bigr) (f')^2 \ge 1
$$
one has $\varphi_{x_1 x_1} \le f(\varphi_{x_1}) \sqrt{\Lambda + \frac{1}{2\delta}}$.
Let us set: $-p=1-\frac{\delta}{2} M^2$.
The proof is complete.
\end{proof}

\begin{corollary}
\label{measure-set}
\begin{itemize}
\item[1)]
Let $\mu = \gamma$ be the standard Gaussian measure. One has
$$
\|D^2 \varphi\| \le \sqrt{d-1} \frac{\mbox{\rm diam(K)}}{4\int_{0}^{\frac{\pi}{2}} \sin^{\frac{4}{d-1}}x \ dx}.
$$
\item[2)]
Let $\mu = \prod_{i=1}^{d} \mu_i$, where every $\mu_i$ is a copy of a measure $\mu_0$ on $\R$ satisfying the following:
a) $\mu_0=e^{-V}dx$ with smooth $V$, b) $|V'| \le 1$, $V'' \le 1$.
For instance, one can choose $V(x)$ to be equal $|x|$ for large values of $|x|$ and quadratic for
small ones. Then
$$
\|D^2 \varphi\| \le
\sqrt{-p\Bigl( 1 + \frac{d}{4(1+p)}\Bigr)} \cdot  \frac{\mbox{\rm diam(K)}}{2\int_{0}^{\frac{\pi}{2}} \sin^{-1-\frac{1}{p}}x \ dx}.
$$
for any $-1 < p <0$.
\end{itemize}
\end{corollary}

\section{Appendix : applications to functional inequalities}

The arguments presented in the following lemma have been communicated to the author by Sasha Sodin.
It turns out that convexity condition allows to avoid the use of semigroup techniques
in Proposition \ref{Stein-lem}.

\begin{lemma}
\label{Sodin-lem}
For every convex $f$ and unit vector $h$  one has
$$
|\nabla f(x+th) -f(x)|
\le \frac{2}{t} \sup_{v: |v|=1}  \Bigl( f(x+2tv) + f(x-2tv) - 2 f(x)
 \Bigr).
$$
\end{lemma}
\begin{proof}
Let us start with $d=1$. Note that
$$
f(x_0+2t) + f(x_0-2t) - 2 f(x_0) =
\int_{-2t}^{0} \bigl( f'(x_0+s+2t) - f'(x_0 + s) \bigr)ds.
$$
Since $f'$ is increasing, one has
$$
f'(x_0+s+2t) - f'(x_0 + s) \ge f'(x_0+t) - f'(x_0)
$$
for every $-t \le s \le 0$. Thus
\begin{align*}
f(x_0+2t) + f(x_0-2t)& - 2 f(x_0)
\ge
\int_{-t}^{0} \bigl( f'(x_0+s+2t) - f'(x_0 + s) \bigr)ds
\\& \ge
\int_{-t}^{0} \bigl(f'(x_0+t) - f'(x_0)\bigr)ds
= t \bigl(f'(x_0+t) - f'(x_0)\bigr).
\end{align*}
This completes the proof of the one-dimensional case.
Let $d\ge 2$. For every couple of unit vectors $u,h$ and $t>0$ one has
by the cyclical monotonicity (see \cite{Vill})
\begin{align*}
 \langle \nabla f(x+tu), th \rangle
+
\langle \nabla f(x), tu \rangle
+ &
\langle \nabla f(x+th), 0 \rangle
\\&
 \le
 \langle \nabla f(x+th), th \rangle
+
\langle \nabla f(x), 0 \rangle
+
\langle \nabla f(x+tu), tu \rangle.
\end{align*}
Hence
$$
 \langle \nabla f(x+tu) -\nabla f(x), th \rangle
 \le
 \langle \nabla f(x+th) - \nabla f(x), th \rangle
+
\langle \nabla f(x+tu) -\nabla f(x), tu \rangle.
$$
Then it follows from the one-dimensional case that
\begin{align*}
 \langle
 & \nabla f(x+tu) - \nabla f(x), th \rangle
 \\&
 \le
 \frac{1}{t}
 \Bigl( f(x+2th) + f(x-2th) - 2 f(x)
 \Bigr)
 +
  \frac{1}{t}
 \Bigl( f(x+2tu) + f(x-2tu) - 2 f(x)
 \Bigr) .
\end{align*}
It remains to take supremum over all unit vectors $h,u$.
\end{proof}

This estimate implies generalized continuity of the optimal transport for other types
of uniform convexity.

\begin{corollary}
\label{MS-conc}
Assume that $W$ satisfies
$$
W(x+y)+W(x-y) - W(x) \ge \delta(|y|)
$$
with some non-negative increasing function $\delta$.
Then
$$
|\nabla \varphi(x) - \nabla \varphi(y) | \le 8 \delta^{-1}(4|x-y|^2).
$$
\end{corollary}
\begin{proof}
Following the proof of Theorem \ref{hoelder}, one easily finds that
$$
\varphi(x+th) + \varphi(x-th) - 2 \varphi(x) \le 2t \delta^{-1}(t^2).
$$
The result follows from Lemma \ref{Sodin-lem}.
\end{proof}

\begin{corollary}
\label{MS}
Under assumptions of Corollary \ref{MS-conc}
 $\nu = e^{-W} dx$
 satisfies
$$
\nu\bigl( A_{r} \bigr) \ge  \Phi\Bigl(\Phi^{-1}(\nu(A)) +  \frac{1}{2} \sqrt{\delta(r/8)} \Bigr),
$$
where $\Phi(t) = \frac{1}{\sqrt{2\pi}} \int_{-\infty}^{t} e^{-\frac{s^2}{2}} \ ds$.
In particular, $\nu$ admits a dimension-free concentration
property
$$
\nu\bigl( A_{r}  \bigr) \ge  1- \frac{1}{2} \exp \Bigl( \frac{1}{8} \ \delta(r/8) \Bigr)
$$
with $\nu(A) \ge 1/2$.
\end{corollary}
\begin{proof}
Consider the optimal mapping $T$
pushing forward the standard Gaussian measure $\gamma = \frac{1}{(2\pi)^{d/2}} e^{-\frac{|x|^2}{2}} dx$
to $\nu$.
The statement follows immediately from
Corollary \ref{MS-conc}, the Gaussian isoperimetric inequality
(see \cite{Bo98}) and the estimate $\Phi(t) \ge 1 - \frac{1}{2} e^{-\frac{t^2}{2}}$.
\end{proof}

Finally, let us note that similar concentration results  can be also derived from the
 so-called above-tangent lemma.
The detailed description, references and various applications can be found in \cite{Vill},
we just mention briefly some important results.
The above-tangent estimate have been used by Talagrand \cite{Tal2} for establishing
transportation inequality for the standard Gaussian measure.
It was understood later that similar arguments
can be used for proving a broad class inequalities of Sobolev type.
In particular, the Bobkov-Ledoux result has been generalized in \cite{CGH}
(see also \cite{Kol04}). The corresponding isoperimetric inequalities
have been proved in \cite{MilSod} by localization arguments.
Finally, transportation approach to functional inequalities for non log-concave measures has been
developed in
\cite{BarKol}.

We note that there exists another  measure of convexity,
which is especially convenient when one deals with the above-tangent arguments.

\begin{remark}
Everywhere below we deal with an {\it arbitrary} (non-Euclidean) norm $\|\cdot\|$ on $\R^d$.
\end{remark}

For a convex $W$ and  $\|\cdot \|$  let us define $\delta: \R^{+} \to \R^{+}$ and $b: \R^{+} \to \R^{+}$  in the following way:
$$
\delta(t) = \inf \Bigl\{ W(x+y) + W(x-y) - 2W(x): \|y\| \ge t \Bigr\}
$$
\begin{equation}
\label{bregman}
b(t) = \inf \Bigl\{ W(x+y)  - W(x) - \langle \nabla W(x), y \rangle : \|y\| \ge t \Bigr.\}
\end{equation}

These quantities are equivalent in a sense.
\begin{lemma}
One has
$$
b(2t) -2 b(t) \ge \delta(t) \ge 2 b(t).
$$
\end{lemma}
\begin{proof}
Relation $\delta(t) \ge 2 b(t)$ follows from
\begin{align*}
W(x+y) & + W(x-y) - 2W(x)
=\\&
\Bigl( W(x+y)  - W(x) - \langle \nabla W(x), y \rangle \Bigr)
+
\Bigl( W(x-y)  - W(x) - \langle \nabla W(x), -y \rangle \Bigr).
\end{align*}
Further, note that
$$
W\Bigl(x+ \frac{y}{2}\Bigr)
-W(x) - \Bigl\langle \nabla W(x), \frac{y}{2} \Bigr\rangle
\ge b(t/2).
$$
Taking $y$ with $|y|=t$ one has
$$
\delta(t/2) \le W(x+y) + W(x) - 2 W\Bigl(x+ \frac{y}{2}\Bigr)
\le
 W(x+y)  - W(x) - \langle \nabla W(x), y \rangle -2 b(t/2).
$$
This clearly implies $\delta(t/2) + 2 b(t/2) \le b(t)$.
\end{proof}

\begin{definition}
In what follows let
\begin{itemize}
\item[1)]
$\tilde{b}$ be the
maximal convex function majorized by $b$
\item[2)]
$b^*$ be the corresponding convex conjugated function:
$$
b^*(t) = \sup_{s>0}\{ ts - b(s)\}.
$$
\end{itemize}
\end{definition}

The results of the following proposition are known (see \cite{CGH}).
For the reader convenience we just sketch the proof of 2).

\begin{proposition}
Let $W$ be a convex function such that
$\nu = e^{-W} dx$ is a probability measure. Consider a norm $\|\cdot\|$ on $\R^d$ and
define $b$ by (\ref{bregman}).
Let   $f \cdot \nu$ be a probability measure and
$T = \nabla \varphi$ be the optimal transportation mapping sending $\nu$ to $f \cdot \nu$.
Then the following holds:
\begin{itemize}
\item[1)]
 Talagrand's type inequality:
$$
 \int_{\R^d} b(\|\nabla \varphi(x) - x\|) d \nu \le \int_{\R^d} f \log f \ d\nu,
$$
\item[2)]
a modified log-Sobolev-type inequality
$$
\int_{\R^d} f \log f d\nu \le \int_{\R^d} b^{*} \Bigl( \Bigl\|  \frac{\nabla f(x)}{f(x)} \Bigr\|_{*}\Bigr) f d\nu,
$$
where $\|\cdot\|_{*}$ is the corresponding dual norm
\end{itemize}
\end{proposition}
\begin{proof}
Let as use the above-tangent arguments.
Let $T = \nabla \varphi$ be the optimal transportation sending $f \cdot \nu$ to
$\nu$. By the change of variables formula
$$
0 = W(x) - W(\nabla \varphi) + \log \det D^2_a \varphi -\log f(x),
$$
where $D^2_a \varphi$ is the second Alexandroff derivative of $\varphi$ (see \cite{Vill} for details).
Hence
\begin{align*}
\log f = &  W(x) - W(\nabla \varphi) + \log \det D^2_a \varphi
=
\\&
W(x) - W(\nabla \varphi) - \langle \nabla W(x), x-\nabla \varphi \rangle
\\&
+\langle \nabla W(x), x-\nabla \varphi \rangle
+ \log \det D^2_a \varphi
\\&
\le - b(\|\nabla \varphi(x) - x\|) + \langle \nabla W(x), x-\nabla \varphi \rangle
+ \log \det D^2_a \varphi.
\end{align*}
Let us integrate this formula  with respect to $f \cdot \nu$. Integration by parts
and convexity properties of $\varphi$
give
\begin{align*}
\int_{\R^d} f \log f d \nu
&
\le
- \int_{\R^d} b(\|\nabla \varphi(x) - x\|) f d\nu
+ \int_{\R^d} \langle \nabla f(x), x - \nabla \varphi \rangle d \nu.
\\&
- \int_{\R^d} \Bigl( \mbox{Tr} D^2_a \varphi - d -  \log \det D^2_a \varphi \Bigr) d\mu.
\end{align*}
It is well-known and easy to verify that $\mbox{Tr} D^2_a \varphi - d -  \log \det D^2_a \varphi \ge 0$.
Applying   $$\Bigl\langle \frac{\nabla f(x)}{f(x)}, x - \nabla \varphi \Bigr\rangle \le
\|x-\varphi(x)\| \cdot \Bigl\|  \frac{\nabla f(x)}{f(x)} \Bigr\|_{*}$$
and  the Young inequality one completes the proof.
\end{proof}

Finally, the following concentration inequality
$$
\nu(A_r) \ge  1- 2 e^{-2 \tilde{b}(r/2)},
$$
$$
 \  \nu(A) \ge \frac{1}{2}, \ A_r = \{y : \|x-y\| \ge r\}
$$
can be easily verified by a modification of Marton's arguments.
Let
 $T_1= \nabla \varphi_1$, $T_2 = \nabla \varphi_2$
be a couple of optimal transportation mappling pushing forward $\nu$ to
$\nu_1 = \frac{1}{\nu(A)} \nu|_{A}$ and $\nu_2 = \frac{1}{\nu(A^c_r)} \nu|_{A^c_r}$
respectively. Then $T = T_2 \circ T_1^{-1}$
sends
$\nu_1$  to $\nu_2$. One has
$$
\tilde{b}(r/2)
\le
\int_{\R^d} \tilde{b}\Bigl(\frac{1}{2} \|x-T(x)\|\Bigr)  d\nu_1
=
\int_{\R^d} \tilde{b}\Bigl(\frac{1}{2} \|T_1-T_2\|\Bigr)  d\nu.
$$
Hence by convexity of $\tilde{b}$ and the Talagrand's type estimate
\begin{align*}
2 \tilde{b}(r/2)
&
\le
\int_{\R^d} \tilde{b}\Bigl(\|T_1(x) - x\|\Bigr)  d\nu
+
\int_{\R^d} \tilde{b}\Bigl( \|T_2(x)-x\|\Bigr)  d\nu
\\&
\le
\log \Bigl( \frac{1}{\nu(A)} \Bigr)
+
\log \Bigl( \frac{1}{\nu(A^c_r)} \Bigr).
\end{align*}
Hence
$$
e^{2 \tilde{b}(r)}
\le \frac{1}{\nu(A) \nu(A^c_r)}.
$$

This work was supported by the RFBR project 07-01-00536,
 GFEN-06-01-39003,
RF President Grant MD-764.2008.1,
DFG Grant 436 RUS 113/343/0(R).

\end{document}